\documentclass[11pt,a4paper]{article}
\usepackage{indentfirst,mathrsfs}
\usepackage{amsfonts,amsmath,amssymb,amsthm}
\usepackage{latexsym,amscd}
\usepackage{amsbsy}
\usepackage{color}
\usepackage[noadjust]{cite}
\usepackage{multirow}
\usepackage{makecell}
\usepackage{float}

\newtheorem{thm}{Theorem}
\newtheorem{lem}{Lemma}
\newtheorem{cor}{Corollary}

\newtheorem{example}{Example}
\newtheorem{remark}{Remark}

\newcommand{\va}{\varepsilon}
\newcommand{\fqt}{{\mathbb F}_{q^2}}
\newcommand{\fq}{{\mathbb F}_{q}}

\begin{document}
\title{On constructing bent functions from cyclotomic mappings}
\author{Xi Xie, Nian Li, Qiang Wang and Xiangyong Zeng
\thanks{X. Xie and X. Zeng are with the Hubei Key Laboratory of Applied Mathematics, Faculty of Mathematics and Statistics, Hubei University, Wuhan 430062, China.
N. Li is with the Hubei Key Laboratory of Applied Mathematics, School of Cyber Science and Technology, Hubei University, Wuhan 430062, China. Q. Wang is with the School of Mathematics and Statistics, Carleton University, Ottawa, K1S 5B6, Canada. 
Email: xi.xie@aliyun.com, nian.li@hubu.edu.cn, wang@math.carleton.ca, xiangyongzeng@aliyun.com.}
}
\date{}%\today
\maketitle
\begin{quote}
{{\bf Abstract:}
We study a new method of constructing Boolean bent functions from cyclotomic mappings.  Three generic constructions are obtained by considering different branch functions such as Dillon functions, Niho functions and Kasami functions over multiplicative cosets and additive cosets respectively. As a result, several new explicit infinite families of bent functions and their duals are derived.
We demonstrate that some previous constructions are special cases of our simple constructions.
In addition, by studying their polynomial forms, we observe that the last construction provides some examples which are EA-inequivalent to five classes of monomials, Dillon type and Niho type polynomials.
}


\end{quote}

\section{Introduction}
Boolean bent functions were first introduced by Rothaus in 1976 \cite{R} as an interesting combinatorial
object with maximum Hamming distance to the set of all affine functions. Over the last four decades, bent
functions have attracted a lot of research interest due to their important applications in cryptography
\cite{C2010}, sequences \cite{OSW} and coding theory \cite{CHLL,DFZ}. Kumar, Scholtz and Welch in \cite{KSW} generalized the notion of Boolean bent functions to the case of functions over an arbitrary finite field.
Then a lot of research has been devoted to the construction of bent functions, which we refer the readers to  \cite{CM}, \cite{Mbook}, and Chapter 6 of \cite{Cbook}. The construction methods can be divided into two categories: primary and secondary constructions. Primary constructions build bent functions from scratch, see, e.g., the papers \cite{CCK,D,Do, Do2,L,LK,Mc,R}. In contrast,  secondary constructions provide bent functions from the known ones, and a non-exhaustive list of references is \cite{C1994,C2004,C2006,C2012,D,LKMPTZ,M,T,X}. Though there are many works on bent functions, the classification and the general structure of bent functions are still not clear. Therefore it remains attractive to continue studying bent functions.

Cyclotomic mappings  of the first order were first introduced by  Evans \cite{Evans}  and   Niederreiter and Winterhof \cite{NW}.  It was further generalized by Wang \cite{W2007, W}. Let $\mathbb{F}_{p^{n}}$ be the finite field with $p^n$ elements. The so-called index $d$ generalized cyclotomic mappings of $\mathbb{F}_{p^n}$ are functions $\mathbb{F}_{p^n}\rightarrow \mathbb{F}_{p^n}$ that agree with a suitable monomial function $x \rightarrow ax^{r}$ (for a fixed $a\in\mathbb{F}_{p^n}$ and non-negative integer $r$) on each cyclotomic coset of the index $d$ subgroup of $\mathbb{F}_{p^n}^*$. It is called a cyclotomic mapping if all exponents of monomials over all cosets are the same exponent.
It turns out that every polynomial fixing $0$ can be represented by a cyclotomic mapping uniquely according to its index \cite{AGW2009}.  Cyclotomic mappings are used extensively to study the permutation behaviour of polynomials over finite fields \cite{AW,W2007, W}. %\cite{AAW2008,AW2005,AW2006,AW2007,W2007,W,Z2008,Z2009,Z2010}.
Recently it was also used to construct linear codes with few weights \cite{FSWW}. In this paper, we explore the new application of cyclotomic mappings in the construction of bent functions.

Let  $q=2^m$ and $m$ be a positive integer. In the first and second generic constructions (see Theorems \ref{thm.bent-dillon} and \ref{thm.bent-Niho}), we use Dillon functions and Niho functions as the branch functions over index $q+1$ cyclotomic cosets of $\fqt$, respectively.
%Therefore two generic constructions of bent functions are proposed (s. %, which provide sufficient and necessary conditions for polynomials of Dillon type and Niho type being bent.
Some explicit classes of bent functions are provided and their duals are determined as well.  Using Magma we also demonstrate that many examples can be derived from these classes.   For the third construction, we introduce a new  variation of cyclotomic mappings. Namely, we partition the finite field $\fqt$ as a union of additive cosets and then use Kasami functions as branch functions over these additive cosets. As a result, we obtain another generic construction of bent functions (see Theorem \ref{thm.bent-kasami}).  Because the conditions in our construction are easy to meet,  we demonstrate our construction by obtaining some explicit classes of non-quadratic bent functions (see Corollaries \ref{cor.kasami0} and \ref{cor.kasami1}).  Finally, switching between the cyclotomic form and polynomial form, we illustrate that our constructions can produce new infinite classes of bent polynomials, as well as  several previous known classes. In fact, these generic constructions produce bent polynomials belonging to the Partial Spread class, the class $\mathcal{H}$ and the Maiorana-McFarland class respectively.

The rest of this paper is organized as follows.  Some preliminaries are given in  Section \ref{prel}.
Section \ref{cons1} and Section \ref{cons2} construct bent functions using Dillon functions, Niho functions and Kasami functions as the branch functions over multiplicative cyclotomic cosets and additive cyclotomic cosets respectively. In Section \ref{sec.poly}, we study the polynomial forms of our newly constructed bent functions and briefly discuss the EA-equivalence between our functions and known ones. Finally, Section \ref{conc} concludes this paper.

\section{Preliminaries}\label{prel}
Throughout this paper, let $\mathbb{Z}_d$ denote the set $\{0,1,\cdots,d-1\}$.
In addition, let $\mathbb{F}_{p^{n}}$ be the finite field with $p^n$ elements and $\mu_e=\{x\in\mathbb{F}_{p^{n}}: x^{e}=1\}$ be the set of $e$-th roots of unity in $\mathbb{F}_{p^{n}}$ for $e\,|\,p^n-1$, where $p$ is a prime and $n$ is a positive integer.
The (absolute) trace function ${\rm Tr}_1^n:\mathbb F_{p^n}\longrightarrow \mathbb F_p$ is defined by ${\rm Tr}_1^n(x)=\sum_{i=0}^{n-1} x^{p^{i}}$ for all $x\in\mathbb F_{p^n}$.

Given a function $f(x)$ mapping from $\mathbb{F}_{p^{n}}$ to $\mathbb{F}_{p}$, the Walsh transform of $f(x)$
is defined by
$$\widehat{f}(b)=\sum\nolimits_{x\in\mathbb{F}_{p^n}}{\xi_p}^{f(x)-{\rm Tr}_1^n(bx)}, \,
b\in\mathbb{F}_{p^n},$$
where $\xi_p=e^{\frac{2\pi\sqrt{-1}}{p}}$ is a complex primitive $p$-th root of unity. Then $f(x)$ is called a $p$-ary bent function if all its Walsh coefficients satisfy
$|\widehat{f}(b) |=p^{n/2}$ \cite{R,KSW}. A $p$-ary bent function $f(x)$ is called regular if
$\widehat{f}(b)=p^{n/2}\omega^{\widetilde{f}(b)}$ holds for some function $\widetilde{f}(x)$ mapping
$\mathbb{F}_{p^n}$ to $\mathbb{F}_p$, and it is called weakly regular if there exists a complex $\mu$ having unit magnitude such that $\widehat{f}(b)=\mu^{-1}p^{n/2}\omega^{\widetilde{f}(b)}$ for all
$b\in\mathbb{F}_{p^n}$. The function $\widetilde{f}(x)$ is called the dual of $f(x)$ and it is also bent.

The following lemmas are helpful for the subsequent sections.

\begin{lem}{\rm(\cite{HK})}\label{lem.km}
Let $n=2m$ and $a\in\fqt^*$, where $q=p^m$ and $p$ is a prime.
Then
 \[\sum\nolimits_{x\in\mu_{q+1}}\xi_p^{{\rm Tr}_1^n(a x)}=1-K_m(a^{q+1}).\]
Here $K_m(a^{q+1}):=\sum_{x\in\fq}\xi_p^{{\rm Tr}_1^m( a^{q+1} x+x^{q-2})}$ is the Kloosterman sum over $\fq$.
 \end{lem}

\begin{lem}{\rm(\cite{TZLH})}\label{lem.eq.x^2}
Let $n=2m$ be an even positive integer and $a,b\in\mathbb{F}_{2^n}^*$ satisfying ${\rm Tr}_1^n(b/a^2)=0$. Then the quadratic equation $x^2+ax+b=0$ has

(1) both two solutions in $\mu_{2^m+1}$ if and only if $b=a^{1-2^m}$ and ${\rm Tr}_1^m(b/a^2)=1$.

(2) exactly one solution in $\mu_{2^m+1}$ if and only if $b\ne a^{1-2^m}$ and
\[(1+b^{2^m+1})(1+a^{2^m+1}+b^{2^m+1})+a^2b^{2^m}+a^{2^m}b=0.\]
 \end{lem}

\section{Bent functions from multiplicative cyclotomic mappings}\label{cons1}
%In this section, we construct bent functions from multiplicative cyclotomic mappings.
Let $p$ be a prime, $d,\,n$ be positive integers such that $d\mid p^n-1$, and $\omega$ be a primitive element of $\mathbb{F}_{p^n}$. Let $C$ be the (unique) index $d$ subgroup of $\mathbb{F}_{p^n}^*$. Then the cosets of $C$ in $\mathbb{F}_{p^n}^*$ are of the form $C_i:=\omega^i C$ for $i\in\mathbb{Z}_{d}$.
It can be seen that $\mathbb{F}_{p^n}=(\bigcup_{i=0}^{d-1}C_i)\bigcup\{0\}$ and $C_i\bigcap C_j=\emptyset$ for $i\ne j$.
Let $(a_0,a_1,\cdots,a_{d-1})\in\mathbb{F}_{p^n}^{d}$ and $r_0,r_1,\cdots,r_{d-1}$ be $d$ non-negative integers.
A generalized cyclotomic mapping \cite{AW,W} of $\mathbb{F}_{p^n}$ of index $d$ is defined as follows:
\begin{equation}\label{F}
F(x)=\left \{\begin{array}{ll}
0, & {\rm if} \,\, x=0,\\
a_i x^{r_i}, &{\rm if} \,\, x\in C_i,\,i\in\mathbb{Z}_{d}.
\end{array} \right.
\end{equation}
Because the finite field is partitioned into the union of $0$ and the multiplicative cosets, we call these cyclotomic mappings  as multiplicative cyclotomic mappings.
In this paper, we investigate the bentness of the function $f(x)={\rm Tr}_1^n(F(x))$.
To study the Walsh transform of $f(x)$ at point $b\in\mathbb{F}_{p^n}$, we first define
\begin{equation}\label{Sij}
S_{i}(b)=\sum\nolimits_{x\in C_i}\xi_{p}^{{\rm Tr}_1^n(a_i x^{r_i})-{\rm Tr}_1^n(bx)}
\end{equation}
for $i\in\mathbb{Z}_{d}$. This gives
\begin{eqnarray}
\widehat{f}(b)
&=&\sum\nolimits_{x \in \mathbb{F}_{p^n}}\xi_p^{{\rm Tr}_1^n(F(x))-{\rm Tr}_1^n(bx)}\nonumber
\\ &=&1+\sum_{i\in\mathbb{Z}_{d}}\sum_{x\in C_i}\xi_{p}^{{\rm Tr}_1^n(a_i x^{r_i})-{\rm Tr}_1^n(bx)}\nonumber
\\ &=&1+\sum\nolimits_{i\in\mathbb{Z}_{d}}S_{i}(b). \label{fb-mul-ge}
\end{eqnarray}
Therefore, to determine $\widehat{f}(b)$, it suffices to calculate $S_{i}(b)$ for all $i\in\mathbb{Z}_{d}$. In this section, we focus on the construction of bent functions from multiplicative cyclotomic mappings of index $q+1$ over $\mathbb{F}_{q^2}$, where $q=p^m$ and $n=2m$.
In this case, $C=\fq^{*}$ and $C_i=\omega^i\fq^*$. Note that $\omega^i\fq^*=\{\omega^{i+k(q+1)}: k \in \mathbb{Z}_{q-1}\}=\{\omega^{i(q-1)(2^{n-1}-1)+(2^{n-1}i+k)(q+1)}: k \in \mathbb{Z}_{q-1}\}=u^i\fq^*$ with the notation $u=\omega^{(q-1)(2^{n-1}-1)}$. Then $S_{i}(b)$ defined by \eqref{Sij} becomes
$$\begin{aligned}S_{i}(b)
&=\sum\nolimits_{x\in u^i\fq^*}\xi_{p}^{{\rm Tr}_1^n(a_i x^{r_i})-{\rm Tr}_1^n(bx)}
\\ &=\sum\nolimits_{y \in \fq^*}\xi_p^{{\rm Tr}_1^n(a_i (u^iy)^{r_i})-{\rm Tr}_1^n(bu^iy)}
\\ &=\sum\nolimits_{y \in \fq^*}\xi_p^{{\rm Tr}_1^m(\alpha_i y^{r_i})-{\rm Tr}_1^m((bu^i+b^qu^{-i})y)}
\end{aligned}$$
with $\alpha_i:=a_iu^{ir_i}+a_i^qu^{-ir_i}$ due to $u\in \mu_{q+1}$.
The general case is difficult to calculate, we shall study $S_{i}(b)$ for the following two cases:

If $r_i\equiv 0 \,({\rm mod}\, q-1)$, then
\begin{eqnarray}
  S_{i}(b)&=&\xi_p^{{\rm Tr}_1^m(\alpha_i)}\sum\nolimits_{y \in \fq^*}\xi_p^{-{\rm Tr}_1^m((bu^i+b^qu^{-i})y)} \nonumber\\
  &=&\left \{\begin{array}{ll}
(q-1)\xi_p^{{\rm Tr}_1^m(\alpha_i)}, & {\rm if} \,\, bu^i+b^qu^{-i}=0,\\
-\xi_p^{{\rm Tr}_1^m(\alpha_i)}, &{\rm otherwise}.
\end{array} \right. \label{Sii-R0}
\end{eqnarray}

If $r_i\equiv p^{t_i} \,({\rm mod}\, q-1)$ for some integer $t_i$, then
$$\begin{aligned}S_{i}(b)
&=\sum\nolimits_{y \in \fq^*}\xi_p^{{\rm Tr}_1^m((\alpha_i^{p^{-t_i}}-(bu^i+b^qu^{-i}))y)}
\\ &=\left \{\begin{array}{ll}
q-1, & {\rm if} \,\, bu^i+b^qu^{-i}=\alpha_i^{p^{-t_i}},\\
-1, &{\rm otherwise},
\end{array} \right.
\end{aligned}$$
which implies that $S_{i}(b)=q T_{i}(b)-1$ with the notation
\begin{equation}\label{Ti-b}
T_{i}(b)=\left \{\begin{array}{ll}
1, & {\rm if} \,\, bu^i+b^qu^{-i}=\alpha_i^{p^{-t_i}},\\
0, &{\rm otherwise}.
\end{array} \right.
\end{equation}
%In this paper, we focus on the bentness of Boolean functions and the $p$-ary functions can be studied in the same manner.

In the sequel, we always assume that $n=2m$, $q=2^m$ and $u=\omega^{(q-1)(2^{n-1}-1)}$ is a primitive element of $\mu_{q+1}$ with $u^{\infty}=0$. Define $R_0:=\{i\in\mathbb{Z}_{q+1}: r_i\equiv 0 \,({\rm mod}\, q-1)\}$ and $R_1:=\{i\in\mathbb{Z}_{q+1}: r_i\equiv p^{t_i} \,({\rm mod}\, q-1)\}$. Then we characterize the Walsh transform of a Boolean function $f(x)$ in the case of $\# R_0+\# R_1=q+1$.

\begin{thm}\label{thm.wf-mul}
Let $r_i$ be $q+1$ non-negative integers satisfying $\# R_0+\# R_1=q+1$ and $a_i\in\fqt$ for $i\in\mathbb{Z}_{q+1}$.
Define
\begin{equation}\label{f-mul}
f(x)={\rm Tr}_1^n(a_ix^{r_i}),\,\,{\rm if} \,\, x\in u^i\mathbb{F}_{q}^*,\,i\in\{\infty\}\cup\mathbb{Z}_{q+1}.
\end{equation}
Then for any $b\in\fqt$,
$$\widehat{f}(b)=\left \{\begin{array}{lll}
q(M_0+\sum_{i\in R_1} T_{i}(b))-(M_0+\# R_1-1), & {\rm if} \,\, b=0,\\
q((-1)^{{\rm Tr}_1^m(\alpha_t)}+\sum_{i\in R_1} T_{i}(b))-(M_0+\# R_1-1), & {\rm if} \,\, b^{q-1}=u^{2t},\,t\in R_0,\\
q\sum_{i\in R_1} T_{i}(b)-(M_0+\# R_1-1), & {\rm if} \,\, b^{q-1}= u^{2t},\,t\in R_1,\\
\end{array} \right.$$
where $M_0:=\sum_{i\in R_0}(-1)^{{\rm Tr}_1^m(\alpha_i)}$ with $\alpha_i:=a_iu^{ir_i}+a_i^qu^{-ir_i}$ and $T_{i}(b)$ is defined by \eqref{Ti-b}.
\end{thm}

\begin{proof}
In this case, for any $b\in\fqt$, \eqref{fb-mul-ge} turns into
\begin{eqnarray}
  \widehat{f}(b)
&=&1+\sum_{i\in\mathbb{Z}_{q+1}}S_{i}(b)=1+\sum_{i\in R_0}S_{i}(b)+\sum_{i\in R_1}S_{i}(b) \nonumber\\
&=&1+\sum_{i\in R_0}S_{i}(b)+\sum_{i\in R_1}(q T_{i}(b)-1) \nonumber\\
&=&\sum_{i\in R_0}S_{i}(b)+q\sum_{i\in R_1} T_{i}(b)-\# R_1+1. \label{fb-mul}
\end{eqnarray}
Next, we calculate $\sum_{i\in R_0}S_{i}(b)$ for $b\in\fqt$ by considering  the following three cases.

\textbf{Case 1}: $b=0$. In this case, \eqref{Sii-R0} gives $S_{i}(b)=(q-1)(-1)^{{\rm Tr}_1^m(\alpha_i)}$ for all $i\in R_0$. This means
$\sum_{i\in R_0}S_{i}(b)=(q-1)M_0$
with the notation $M_0=\sum_{i\in R_0}(-1)^{{\rm Tr}_1^m(\alpha_i)}$.

\textbf{Case 2}: $b^{q-1}=u^{2t}$ for some $t\in R_0$. If this case happens, then $b^qu^{-t}+bu^t=0$. Further, \eqref{Sii-R0} implies $S_{t}(b)=(q-1)(-1)^{{\rm Tr}_1^m(\alpha_t)}$. Next we claim that $b^qu^{-i}+bu^i\ne0$ for any $i\in R_0\backslash\{t\}$. Suppose that there exists $t'\in R_0\backslash\{t\}$ such that $b^qu^{-t'}+bu^{t'}=0$. Then $u^{2t'}=b^{q-1}=u^{2t}$, i.e., $u^{2(t-t')}=1$, which is impossible due to $\gcd(2,\,q+1)=1$ and $t\ne t'$. Thus $b^qu^{-i}+bu^i\ne0$ for any $i\in R_0\backslash\{t\}$. From \eqref{Sii-R0} we can know that $S_{i}(b)=-(-1)^{{\rm Tr}_1^m(\alpha_i)}$ for $i\in R_0\backslash\{t\}$. Hence one concludes
\[\sum_{i\in R_0}S_{i}(b)=(q-1)(-1)^{{\rm Tr}_1^m(\alpha_t)}-\sum_{i\in R_0\backslash\{t\}}(-1)^{{\rm Tr}_1^m(\alpha_i)}=q(-1)^{{\rm Tr}_1^m(\alpha_t)}-M_0.\]

\textbf{Case 3}: $b^{q-1}=u^{2t}$ for some $t\in R_1$. If this case happens, $b^{q-1}\ne u^{2i}$, i.e., $b^qu^{-i}+bu^i\ne0$ for all $i\in R_0$. Then \eqref{Sii-R0} yields $S_{i}(b)=-(-1)^{{\rm Tr}_1^m(\alpha_i)}$ for $i\in R_0$. Therefore we have $\sum_{i\in R_0}S_{i}(b)=-M_0$.

Substituting the values of $\sum_{i\in R_0}S_{i}(b)$ in Cases 1-3 into \eqref{fb-mul} gives the desired result.
\end{proof}

Next we consider the construction of bent functions from two classes of multiplicative cyclotomic mappings of index $q+1$ over $\mathbb{F}_{q^2}$.

\subsection{Bent functions from Dillon functions}\label{cons1.1}
First of all, we focus on the branch functions with Dillon exponents. In this case, we consider $R_0=\mathbb{Z}_{q+1}$ and $R_1=\emptyset$. And we assume $q>2$ in this subsection. Then we state our first result on bent functions.

\begin{thm}\label{thm.bent-dillon}
Let $a_i\in\fqt$ and $l_i$ be non-negative integers for $i\in\mathbb{Z}_{q+1}$. Then
\[f(x)={\rm Tr}_1^n(a_ix^{l_i(q-1)}),\,\,{\rm if} \,\, x\in u^i\mathbb{F}_{q}^*,\,i\in\{\infty\}\cup\mathbb{Z}_{q+1}
\]
is bent if and only if
\[\sum\nolimits_{i\in\mathbb{Z}_{q+1}}(-1)^{{\rm Tr}_1^n(a_iu^{-2il_i})}=1.\]
Moreover, the dual function of $f(x)$ is
$$\widetilde{f}(x)={\rm Tr}_1^n(a_ix^{l_i(1-q)}),\,\, {\rm if} \,\, x^{q-1}=u^{2i},\,i\in\{\infty\}\cup\mathbb{Z}_{q+1}.$$
\end{thm}

\begin{proof}
From Theorem \ref{thm.wf-mul}, for any $b\in\fqt$, one has
\begin{equation}\label{fb-Dillon-ge}
\widehat{f}(b)=\left \{\begin{array}{lll}
qM_0-(M_0-1), & {\rm if} \,\, b=0,\\
q(-1)^{{\rm Tr}_1^m(\alpha_t)}-(M_0-1), & {\rm if} \,\, b^{q-1}=u^{2t},\,t\in\mathbb{Z}_{q+1}\\
\end{array} \right.
\end{equation}
due to  that $\# R_1=0$ and  $\sum_{i\in R_1} T_{i}(b)=0$. Here $\alpha_i=a_iu^{il_i(q-1)}+a_i^qu^{-il_i(q-1)}=a_iu^{-2il_i}+a_i^qu^{2il_i}$ and $M_0=\sum_{i\in\mathbb{Z}_{q+1}}(-1)^{{\rm Tr}_1^m(\alpha_i)}=\sum_{i\in\mathbb{Z}_{q+1}}(-1)^{{\rm Tr}_1^n(a_iu^{-2il_i})}$. Note that when $i$ runs over $\mathbb{Z}_{q+1}$, $u^{2i}$ runs over $\mu_{q+1}$, which implies $\{0\}\bigcup\{b\in\fqt: b^{q-1}=u^{2i}, i\in \mathbb{Z}_{q+1}\}=\fqt$. Then we claim that $|\widehat{f}(b)|=q$ for any $b\in\fqt$ if and only if $M_0=1$. Obviously, if $M_0=1$, then \eqref{fb-Dillon-ge} yields $|\widehat{f}(b)|=q$ for any $b\in\fqt$. On the other hand, if $|\widehat{f}(b)|=q$ for any $b\in\fqt$, then $|\widehat{f}(0)|=q$ and then \eqref{fb-Dillon-ge} indicates $\widehat{f}(0)=qM_0-(M_0-1)=(q-1)M_0+1=\pm q$, which leads to $M_0=1$ since $q>2$.
Thus $|\widehat{f}(b)|=q$ for any $b\in\fqt$ if and only if $M_0=1$.
More precisely, \eqref{fb-Dillon-ge} gives $\widehat{f}(0)=q$, and if $b^{q-1}=u^{2t}$ for some $t\in\mathbb{Z}_{q+1}$, then
\[\widehat{f}(b)=q(-1)^{{\rm Tr}_1^m(\alpha_t)}=q(-1)^{{\rm Tr}_1^n(a_tu^{-2tl_t})}=q(-1)^{{\rm Tr}_1^n(a_tb^{l_t(1-q)})}.\]
This completes the proof.
\end{proof}

A slight modification of the condition given in Theorem \ref{thm.bent-dillon} such that $f(x)$ is bent allows us to construct bent functions explicitly. Note that
\[\sum_{i\in\mathbb{Z}_{q+1}}(-1)^{{\rm Tr}_1^n(a_iu^{-2il_i})}=\sum_{i\in\mathbb{Z}_{q}}((-1)^{{\rm Tr}_1^n(a_iu^{-2il_i})}-(-1)^{{\rm Tr}_1^n(a_qu^{-2il_q})})+\sum_{i\in\mathbb{Z}_{q+1}}(-1)^{{\rm Tr}_1^n(a_qu^{-2il_q})}\]
  and when $\gcd(l_q,\,q+1)=1$ and $a_q\ne 0$,
\[\sum_{i\in\mathbb{Z}_{q+1}}(-1)^{{\rm Tr}_1^n(a_qu^{-2il_q})}=\sum_{z\in\mu_{q+1}}(-1)^{{\rm Tr}_1^n(a_qz)}=1-K_m(a_q^{q+1})\]
due to Lemma \ref{lem.km}. That means $f(x)$ in Theorem \ref{thm.bent-dillon} is bent if and only if
\[\sum\nolimits_{i\in\mathbb{Z}_{q}}((-1)^{{\rm Tr}_1^n(a_iu^{-2il_i})}-(-1)^{{\rm Tr}_1^n(a_qu^{-2il_q})})=K_m(a_q^{q+1}).\]
By using this observation, we obtain the following construction of bent functions for the case $\#\{a_i: i\in\mathbb{Z}_{q+1}\}\leq 2$ directly.

\begin{thm}\label{thm.bent-dillon2}
Let $a_1\in\fqt,a_2\in\fqt^*$, $l_1,\,l_2$ be non-negative integers and $\gcd(l_2,\,q+1)=1$. Denote $N:=\{u^{i}\fq^*: i \in\mathbb{Z} \}$, where $\mathbb{Z}$ is a subset of $\mathbb{Z}_{q}$. Then
$$f(x)=\left \{\begin{array}{ll}
{\rm Tr}_1^n(a_1 x^{l_1(q-1)}), & {\rm if} \,\, x\in N,\\
{\rm Tr}_1^n(a_2 x^{l_2(q-1)}), &{\rm otherwise}
\end{array} \right.$$
is a bent function if and only if
\begin{equation}\label{con.dillon1}\sum\nolimits_{i\in\mathbb{Z}}((-1)^{{\rm Tr}_1^n(a_1u^{-2il_1})}-(-1)^{{\rm Tr}_1^n(a_2u^{-2il_2})})=K_m(a_2^{q+1}).
\end{equation}
Moreover, the dual function of $f(x)$ is
$$\widetilde{f}(x)=\left \{\begin{array}{ll}
{\rm Tr}_1^n(a_1x^{l_1(1-q)}), & {\rm if} \,\, x^{q-1}=u^{2i},\,i\in\mathbb{Z},\\
{\rm Tr}_1^n(a_2x^{l_2(1-q)}), &{\rm otherwise} .
\end{array} \right.$$
\end{thm}

It is easy to construct bent functions by selecting suitable set $N$ and parameters $a_1,a_2$ when $\gcd(l_1,\,q+1)=\gcd(l_2,\,q+1)=1$. In the case of $N=\mu_{r(q-1)}$ with $r\mid (q+1)$, one has $\mathbb{Z}=\{i\in\mathbb{Z}_{q}:(q+1)/r \mid i \}$. Then \eqref{con.dillon1} becomes
\[\sum\nolimits_{i\in\mathbb{Z}_{q},(q+1)/r|i}((-1)^{{\rm Tr}_1^n(a_1 u^{-2il_1})}-(-1)^{{\rm Tr}_1^n(a_2 u^{-2il_2})})=K_m(a_2^{q+1}),\]
that is,
\begin{equation}\label{eq.dillon2}
\sum\nolimits_{i\in\mathbb{Z}_{r}}((-1)^{{\rm Tr}_1^n(a_1\va^{i})}-(-1)^{{\rm Tr}_1^n(a_2\va^{i})})=K_m(a_2^{q+1})
\end{equation}
due to $\gcd(l_i,\,q+1)=1$ for $i=1,\,2$, where $\va$ is a primitive element of $\mu_{r}$.
Firstly, set $a_1=\va^{j} a_2$ for $j\in \mathbb{Z}_{r}$, then
\[\sum\nolimits_{i\in\mathbb{Z}_{r}}(-1)^{{\rm Tr}_1^n(a_1\va^{i})}=\sum\nolimits_{i\in\mathbb{Z}_{r}}(-1)^{{\rm Tr}_1^n(a_2\va^{i+j})}=\sum\nolimits_{i\in\mathbb{Z}_{r}}(-1)^{{\rm Tr}_1^n(a_2\va^{i})}.\]
Then the following result can be obtained from Theorem \ref{thm.bent-dillon2} and \eqref{eq.dillon2} directly.

\begin{cor}\label{cor.bent-dillon3.1}
Let $c\in\fqt^*$, $\epsilon\in\mu_{r}$, $r$ and $l_i$ be positive integers satisfying $r|(q+1)$ and $\gcd(l_i,\,q+1)=1$ for $i=1,\,2$. Define
$$f(x)=\left \{\begin{array}{ll}
{\rm Tr}_1^n(\epsilon c x^{l_1(q-1)}), & {\rm if} \,\, x\in \mu_{r(q-1)},\\
{\rm Tr}_1^n(c x^{l_2(q-1)}), &{\rm otherwise}.
\end{array} \right.$$
Then $f(x)$ is a bent function if and only if $K_m(c^{q+1})=0$. Moreover, the dual function of $f(x)$ is
$$\widetilde{f}(x)=\left \{\begin{array}{ll}
{\rm Tr}_1^n(\epsilon c x^{l_1(1-q)}), & {\rm if} \,\, x\in \mu_{r(q-1)},\\
{\rm Tr}_1^n(cx^{l_2(1-q)}), &{\rm otherwise} .
\end{array} \right.$$
\end{cor}

\begin{remark}
If we use $l_1=l_2$ and $\epsilon=1$ in Corollary \ref{cor.bent-dillon3.1}, then $f(x)$ is reduced to the monomial case, and Corollary \ref{cor.bent-dillon3.1} gives ${\rm Tr}_1^n(c x^{l_1(q-1)})$ is bent if and only if $K_m(c^{q+1})=0$, which is consistent with the results presented by Dillon \cite{D} for $l_1=1$, Leander \cite{L} and Charpin and Gong \cite{CG} for $l_1$ with $\gcd(l_1,\,q+1)=1$. By the way, $f(x)$ in Corollary \ref{cor.bent-dillon3.1}  is bent if and only if the function ${\rm Tr}_1^n(c x^{l_2(q-1)})$ is bent.
%Dillon characterized the bentness of ${\rm Tr}_1^n(c x^{l_1(q-1)})$ for $l_1=1$, and Lender and Charpin and Gong generalized it to any $l_1$ with $\gcd(l_1,\,q+1)=1$. Our result is consistent with
\end{remark}

\begin{example}
{\rm Let $q=2^6$, $l_1=l_2=1$, $r=5$. According to Magma, there are 3118 pairs $(\epsilon ,\,c)$ with $\epsilon\ne 1$ such that $f(x)$ in Corollary \ref{cor.bent-dillon3.1} is bent over $\mathbb{F}_{2^{12}}$. Take $\epsilon=\omega^{819}$ and $c=\omega^{5}$, where $\omega$ is a primitive element of $\mathbb{F}_{2^6}$. Then $\epsilon\in\mu_{5}$ and it can be checked that $K_6(\omega^{65})=0$. Corollary \ref{cor.bent-dillon3.1} now establishes that
$$f(x)=\left \{\begin{array}{ll}
{\rm Tr}_1^{12}(\omega^{824}x^{63}), & {\rm if} \,\, x\in \mu_{315},\\
{\rm Tr}_1^{12}(\omega^{5}x^{63}), &{\rm otherwise}
\end{array} \right.$$
is a bent function over $\mathbb{F}_{2^{12}}$ and its dual is
$$\widetilde{f}(x)=\left \{\begin{array}{ll}
{\rm Tr}_1^{12}(\omega^{824} x^{-63}), & {\rm if} \,\, x\in \mu_{315},\\
{\rm Tr}_1^{12}(\omega^{5} x^{-63}), &{\rm otherwise} .
\end{array} \right.$$
}
\end{example}

Secondly, by setting $r=3$, we give the following result.
\begin{cor}\label{cor.bent-dillon3.2}
Let $q=2^m$ for an odd integer $m>1$, $a_1\in\fqt$, $a_2\in\fqt^*$ and $l_i$ be positive integers satisfying $\gcd(l_i,\,q+1)=1$ for $i=1,\,2$. Then
$$f(x)=\left \{\begin{array}{ll}
{\rm Tr}_1^n(a_1x^{l_1(q-1)}), & {\rm if} \,\, x\in \mu_{3(q-1)},\\
{\rm Tr}_1^n(a_2x^{l_2(q-1)}), &{\rm otherwise}
\end{array} \right.$$
is a bent function if and only if
\[K_m(a_{2}^{q+1})=4\big((1-{\rm Tr}_1^n(a_1))(1-{\rm Tr}_1^n(a_1\theta))-(1-{\rm Tr}_1^n(a_2))(1-{\rm Tr}_1^n(a_2\theta))\big),\]
where $\theta$ is a 3-rd root of unity. Moreover, the dual function of $f(x)$ is
$$\widetilde{f}(x)=\left \{\begin{array}{ll}
{\rm Tr}_1^n(a_1 x^{l_1(1-q)}), & {\rm if} \,\, x\in \mu_{3(q-1)},\\
{\rm Tr}_1^n(a_2 x^{l_2(1-q)}), &{\rm otherwise} .
\end{array} \right.$$
\end{cor}

\begin{proof}
According to Theorem \ref{thm.bent-dillon2} and \eqref{eq.dillon2}, $f(x)$ is bent if and only if
\begin{equation}\label{eq.dillon3.2}
\sum_{i=0}^{2}((-1)^{{\rm Tr}_1^n(a_1\theta^{i})}-(-1)^{{\rm Tr}_1^n(a_{2}\theta^{i})})=K_m(a_{2}^{q+1}),
\end{equation}
where $\theta$ is a $3$-rd root of unity. Note that $\theta^2+\theta+1=0$, which means ${\rm Tr}_1^n(a_j\theta^{2})={\rm Tr}_1^n(a_j)+{\rm Tr}_1^n(a_j\theta)$ for $j=1,\,2$. Observe that for two Boolean functions $f_1(x)$ and $f_2(x)$,
\[(-1)^{f_1(x)}+(-1)^{f_2(x)}+(-1)^{f_1(x)+f_2(x)}=4(1-f_1(x))(1-f_2(x))-1.\]
This implies
\[\sum_{i=0}^{2}(-1)^{{\rm Tr}_1^n(a_j\theta^{i})}=4(1-{\rm Tr}_1^n(a_j))(1-{\rm Tr}_1^n(a_j\theta))-1\]
for $j=1,\,2$. The desired result then follows from \eqref{eq.dillon3.2}. This completes the proof.
\end{proof}

\begin{example}
{\rm Let $q=2^5$, $l_1=l_2=1$. According to Magma, there are 218715 pairs $(a_1,\,a_2)$ with $a_1\ne a_2$ such that $f(x)$ in Corollary \ref{cor.bent-dillon3.2} is bent over $\mathbb{F}_{2^{10}}$.  For example, take $a_1=1$ and $a_2=\omega^{9}$, where $\omega$ is a primitive element of $\mathbb{F}_{2^{10}}$. Then $\omega^{341}$ is a 3-rd root of unity. It can be verified that $K_5(\omega^{33})=-4$, ${\rm Tr}_1^n(1)=0$, ${\rm Tr}_1^n(\omega^{341})=1$, ${\rm Tr}_1^n(\omega^{9})={\rm Tr}_1^n(\omega^{350})=0$. Corollary \ref{cor.bent-dillon3.2} now establishes that
$$f(x)=\left \{\begin{array}{ll}
{\rm Tr}_1^{10}(x^{31}), & {\rm if} \,\, x\in \mu_{93},\\
{\rm Tr}_1^{10}(\omega^{9}x^{31}), &{\rm otherwise}
\end{array} \right.$$
is a bent function over $\mathbb{F}_{2^{10}}$ and its dual is
$$\widetilde{f}(x)=\left \{\begin{array}{ll}
{\rm Tr}_1^{10}(x^{-31}), & {\rm if} \,\, x\in \mu_{93},\\
{\rm Tr}_1^{10}(\omega^{9}x^{-31}), &{\rm otherwise} .
\end{array} \right.$$}
\end{example}

\subsection{Bent functions from Niho functions}\label{cons1.2}
Secondly, we focus on the branch functions with Niho exponents. Without loss of generality, we assume that $r_i=s_i(q-1)+1$ for all $i\in\mathbb{Z}_{q+1}$. In this case, we consider $R_0=\emptyset$ and $R_1=\mathbb{Z}_{q+1}$. Then the second main result on bent functions follows from Theorem \ref{thm.wf-mul}.

\begin{thm}\label{thm.bent-Niho}
Let $a_i\in\fqt$ and $s_i$ be integers with $0\leq s_i\leq q$ for $i\in\mathbb{Z}_{q+1}$. Define
\[f(x)={\rm Tr}_1^n(a_ix^{s_i(q-1)+1}),\,\,{\rm if} \,\, x\in u^i\mathbb{F}_{q}^*,\,i\in\{\infty\}\cup\mathbb{Z}_{q+1}.
\]
Then for any $b\in\fqt$, $\widehat{f}(b)=q(\sum_{i\in\mathbb{Z}_{q+1}}T_{i}(b)-1)$, where
\begin{equation}\label{Ti-b-Niho}
T_{i}(b)=\left \{\begin{array}{ll}
1, & {\rm if} \,\, bu^i+b^qu^{-i}=\alpha_i,\\
0, &{\rm otherwise}.
\end{array} \right.
\end{equation}
with $\alpha_i=a_iu^{i(1-2s_i)}+a_i^qu^{i(2s_i-1)}$. Moreover,
$f(x)$ is bent if and only if $\sum_{i\in\mathbb{Z}_{q+1}}T_{i}(b)=0$ or $2$.
\end{thm}

We first give a simple case of above construction.

\begin{thm}\label{thm.bent-Niho1}
Let $a_i\in\fqt$ and $s_i$ be integers with $0\leq s_i\leq q$ satisfying $a_iu^{i(1-2s_i)}+a_i^qu^{i(2s_i-1)}=c$ for any $i\in\mathbb{Z}_{q+1}$, where $c$ is a fixed element in $\fq^*$. Then $f(x)$ defined as in Theorem \ref{thm.bent-Niho} is bent and its dual is ${\rm Tr}_1^m(c^{-2}x^{q+1})+1$.
\end{thm}

\begin{proof}
In this case, one knows $\alpha_i=c$ for any $i\in\mathbb{Z}_{q+1}$, and then for $b\in\fqt$, $T_{i}(b)$ given by \eqref{Ti-b-Niho} turns into
$$T_{i}(b)=\left \{\begin{array}{ll}
1, & {\rm if} \,\, bu^i+b^qu^{-i}=c,\\
0, &{\rm otherwise}.
\end{array} \right.$$
Next we determine $T_{i}(b)$ for any $b\in\fqt$. It can be verified that $T_i(0)=0$ for any $i\in\mathbb{Z}_{q+1}$ due to $c\ne 0$. Hence $\sum_{i\in\mathbb{Z}_{q+1}}T_i(0)=0$ and then $\widehat{f}(0)=-q$ according to Theorem \ref{thm.bent-Niho}. For $b\in\fqt^*$, $T_{i}(b)=1$ if and only if $u^{2i}+b^{-1}c u^i+b^{q-1}=0$. Note that the equation
\[z^2+b^{-1}c z+b^{q-1}=0\]
has 2 (resp. 0) solutions in $\mu_{q+1}$ if ${\rm Tr}_1^m(\frac{b^{q-1}}{(b^{-1}c)^2})={\rm Tr}_1^m(b^{q+1}c^{-2})=1$ (resp. ${\rm Tr}_1^m(b^{q+1}c^{-2})=0$);  indeed,  this follows from Lemma \ref{lem.eq.x^2} because ${\rm Tr}_1^n(\frac{b^{q-1}}{(b^{-1}c)^2})=0$ and $b^{q-1}=(b^{-1}c)^{1-q}$ due to $c\in\fq^*$.
From this fact we can derive that there are 2 or 0 $T_i$'s equal to 1 and others $T_i=0$. By  Theorem \ref{thm.bent-Niho},
$\widehat{f}(b)=q(\sum_{i\in\mathbb{Z}_{q+1}}T_{i}(b)-1)=q$ if ${\rm Tr}_1^m(b^{q+1}c^{-2})=1$, or $\widehat{f}(b)=q(\sum_{i\in\mathbb{Z}_{q+1}}T_{i}(b)-1)=-q$ if ${\rm Tr}_1^m(b^{q+1}c^{-2})=0$.  Together with $\widehat{f}(0)=-q$, we conclude $\widehat{f}(b)=q(-1)^{{\rm Tr}_1^m(b^{q+1}c^{-2})+1}$ for any $b\in\fqt$. This completes the proof.
\end{proof}

\begin{remark}
It can be seen that the bent function $f(x)$ given in Theorem \ref{thm.bent-Niho1} is a dual of the Kasami bent function. On the other hand, set $a_i=a$ with $a\in\fqt\backslash\fq$ and $s_i=2^{m-1}+1$ for all $z\in\mathbb{Z}_{q+1}$, then $a_iu^{i(1-2s_i)}+a_i^qu^{i(2s_i-1)}=a+a^q=c\in\fq^*$. In this case, $f(x)={\rm Tr}_1^n(ax^{2^{-1}(q+1)})={\rm Tr}_1^m(c^2 x^{q+1})$ and Theorem \ref{thm.bent-Niho1} implies that $f(x)$ is bent and its dual is ${\rm Tr}_1^m(c^{-2}x^{q+1})+1$. This coincides with the result obtained by Mesnager in \cite{M}.
\end{remark}

In other cases, we obtain a general result.

\begin{thm}\label{thm.bent-Niho2}
Let $a_i\in\fqt$, $s_i$ be integers with $0\leq s_i\leq q$ and denote $\alpha_i=a_iu^{(1-2s_i)i}+a_i^qu^{(2s_i-1)i}$, where $i\in\mathbb{Z}_{q+1}$. Then
$f(x)$ defined as in Theorem \ref{thm.bent-Niho} is bent if and only if one of the following conditions is satisfied:

(1) $\alpha_i\ne 0$ for all $i\in\mathbb{Z}_{q+1}$, and for any $t\in\mathbb{Z}_{q+1}$, each element in the multiset $\{\{\frac{\alpha_i}{u^i+u^{2t-i}}: i\in\mathbb{Z}_{q+1}\backslash\{t\}\}\}$ has multiplicity 2.

(2) $\alpha_{i_1}=\alpha_{i_2}= 0$ for two distinct integers $i_1,i_2\in\mathbb{Z}_{q+1}$ and $\alpha_i\ne 0$ for $i\in\mathbb{Z}_{q+1}\backslash\{i_1,\,i_2\}$; all elements in the set $\{\frac{\alpha_i}{u^i+u^{2t-i}}: i \in \mathbb{Z}_{q+1}\backslash\{i_1,\,i_2\}\}$ are distinct for $t=i_1,i_2$; and each element in the multiset $\{\{\frac{\alpha_i}{u^i+u^{2t-i}}: i\in\mathbb{Z}_{q+1}\backslash\{t,\,i_1,\,i_2\}\}\}$ has multiplicity 2 for any $t\in\mathbb{Z}_{q+1}\backslash\{i_1,\,i_2\}$.
\end{thm}

\begin{proof}
From Theorem \ref{thm.bent-Niho}, to prove $f(x)$ is bent, it suffices to prove $\sum_{i\in\mathbb{Z}_{q+1}}T_{i}(b)=0$ or $2$ for any $b\in\fqt$, where $T_i(b)$ is defined by \eqref{Ti-b-Niho}. We shall distinguish the cases $b=0$ and $b\ne0$ as follows.

If $b=0$, then \eqref{Ti-b-Niho} yields $T_i(0)=1$ if and only if $\alpha_i=0$. This indicates $\sum_{i\in\mathbb{Z}_{q+1}}T_{i}(0)=\#\{i\in\mathbb{Z}_{q+1}: \alpha_i=0\}.$
Thus $f(x)$ is bent only if there are exactly 0 or 2 $\alpha_i$'s equal 0.

If $b^{q-1}=u^{2t}$ for some $t\in\mathbb{Z}_{q+1}$, then $T_i(b)$ given by \eqref{Ti-b-Niho} turns into
\begin{equation}\label{Ti-b-Niho-1}
T_{i}(b)=\left \{\begin{array}{ll}
1, & {\rm if} \,\, b(u^i+u^{2t-i})=\alpha_i,\\
0, &{\rm otherwise}.
\end{array} \right.
\end{equation}
Note that $u^i+u^{2t-i}=0$ if and only if $i=t$.
We consider the value of $\sum_{i\in\mathbb{Z}_{q+1}}T_{i}(b)$ as follows:

(1) $\alpha_i\ne 0$ for any $i\in \mathbb{Z}_{q+1}$. If this case happens, one readily knows $T_{t}(b)=0$ and
 \begin{equation}\label{Ti-b-1}
 T_{i}(b)=\left \{\begin{array}{ll}
1, & {\rm if} \,\, b=\frac{\alpha_i}{u^i+u^{2t-i}},\\
0, &{\rm otherwise}
\end{array} \right.
\end{equation}
for $i\in\mathbb{Z}_{q+1}\backslash\{t\}$ according to \eqref{Ti-b-Niho-1}. Then we can deduce that
$\sum_{i\in\mathbb{Z}_{q+1}}T_{i}(b)=\#\{i\in\mathbb{Z}_{q+1}\backslash\{t\}: \frac{\alpha_i}{u^i+u^{2t-i}}=b\}$, which implies that $\sum_{i\in\mathbb{Z}_{q+1}}T_{i}(b)=0$ or $2$ if and only if each element in the multiset $\{\{\frac{\alpha_i}{u^i+u^{2t-i}}: i\in\mathbb{Z}_{q+1}\backslash\{t\}\}\}$ has multiplicity 2.

(2) There are two distinct integers $i_1,\, i_2\in\mathbb{Z}_{q+1}$ such that $\alpha_{i_1}=\alpha_{i_2}= 0$.
If $t=i_1$, then \eqref{Ti-b-Niho-1} yields $T_{i_1}(b)=1$, $T_{i_2}(b)=0$, and other $T_i(b)$'s are given by \eqref{Ti-b-1} due to $u^i+u^{2t-i}\ne0$ for $i\ne t$. Hence $\sum_{i\in\mathbb{Z}_{q+1}}T_{i}(b)=1+\#\{i\in\mathbb{Z}_{q+1}\backslash\{i_1,\,i_2\}: \frac{\alpha_i}{u^i+u^{2t-i}}=b\}>0$. This implies $\#\{i\in\mathbb{Z}_{q+1}\backslash\{i_1,\,i_2\}: \frac{\alpha_i}{u^i+u^{2t-i}}=b\}=1$. It can be verified that $(\frac{\alpha_i}{u^i+u^{2t-i}})^{q-1}=u^{2t}$ and there are $q-1$ $b$'s such that $b^{q-1}=u^{2t}$ for any $b\in\fqt^*$. Thus $\#\{i\in\mathbb{Z}_{q+1}\backslash\{i_1,\,i_2\}: \frac{\alpha_i}{u^i+u^{2t-i}}=b\}=1$ for any $b^{q-1}=u^{2t}$ if and only if all elements in the set $\{\frac{\alpha_i}{u^i+u^{2i_1-i}}: i \in \mathbb{Z}_{q+1}\backslash\{i_1,\,i_2\}\}$ are distinct.
 The case $t=i_2$ is similar.

 If $t\ne i_1$ and $t\ne i_2$, then \eqref{Ti-b-Niho-1} indicates $T_{t}(b)=T_{i_1}(b)=T_{i_2}(b)=0$ and other $T_i(b)$'s are given by \eqref{Ti-b-1}. Thus $\sum_{i\in\mathbb{Z}_{q+1}}T_{i}(b)=\#\{i\in\mathbb{Z}_{q+1}\backslash\{t,\,i_1,\,i_2\}: \frac{\alpha_i}{u^i+u^{2t-i}}=b\}$, which implies that $\sum_{i\in\mathbb{Z}_{q+1}}T_{i}(b)=0$ or $2$ if and only if each element in the multiset $\{\{\frac{\alpha_i}{u^i+u^{2t-i}}: i\in\mathbb{Z}_{q+1}\backslash\{t,\,i_1,\,i_2\}\}\}$ has multiplicity 2. This completes the proof.
\end{proof}

We don't have an efficient way to construct the explicit classes of such bent functions. However, we can search for many bent examples by Magma, and some of them are given below.

\begin{example}\label{eq.Niho}
{\rm Let $q=2^3$. According to Magma, it can be verified that $s_i=4$ for all $i\in\mathbb{Z}_9$, $a_i=1$ for $i=0,2,4$, $a_7=\omega^5$ and others $a_i=\omega^{2}$; $s_i=2$ for all $i\in\mathbb{Z}_9$, $a_i=1$ for $i=2,4,7,8$, $a_7=\omega^6$ for $i=1,3$ and others $a_i=\omega^{2}$
satisfy the conditions (1) and (2) in Theorem \ref{thm.bent-Niho2} respectively. Then Theorem \ref{thm.bent-Niho2} establishes that
$$f(x)=\left \{\begin{array}{lll}
{\rm Tr}_1^6( x^{29}), & {\rm if} \,\, x\in w^{7i}\fq^*,\,\,i \in\{0,2,4\},\\
{\rm Tr}_1^6(\omega^5 x^{29}), & {\rm if} \,\, x\in w^{7i}\fq^*,\,\,i \in\{7\},\\
{\rm Tr}_1^6(\omega^{2}x^{29}), &{\rm otherwise}
\end{array} \right.$$
and
$$f(x)=\left \{\begin{array}{lll}
{\rm Tr}_1^6( x^{15}), & {\rm if} \,\, x\in w^{7i}\fq^*,\,\,i \in\{2,4,7,8\},\\
{\rm Tr}_1^6(\omega^6 x^{15}), & {\rm if} \,\, x\in w^{7i}\fq^*,\,\,i \in\{1,3\},\\
{\rm Tr}_1^6(\omega^{2}x^{15}), &{\rm otherwise}
\end{array} \right.$$
are bent functions over $\mathbb{F}_{2^6}$.}
\end{example}

\section{Bent functions from additive cyclotomic mappings}\label{cons2}
In this section, we construct bent functions from additive cyclotomic mappings.
Let $p$ be a prime and $d,\,n$ be positive integers such that $d \mid p^n$.
Let $C$ be the index $d$ subgroup of the additive group of $\mathbb{F}_{p^n}$. Then the left cosets of $\mathbb{F}_{p^n}$ modulo $C$ are of the form $C_i:=\upsilon_i+C$ for $i\in\mathbb{Z}_{d}$, where $\upsilon_i$ is a representative of $C_i$. It can be seen that $\mathbb{F}_{p^n}=\bigcup_{i=0}^{d-1}C_i$ and $C_i\bigcap C_j=\emptyset$ for $i\ne j$.
Let $(a_0,a_1,\cdots,a_{d-1})\in\mathbb{F}_{p^n}^{d}$ and $r_0,r_1,\cdots,r_{d-1}$ be $d$ non-negative integers. Define a cyclotomic mapping of $\mathbb{F}_{p^n}$ of index $d$ as follows:
\begin{equation}\label{F.add}
F(x)=a_i x^{r_i}, {\rm if} \,\, x\in C_i,\,i\in\mathbb{Z}_{d},
\end{equation}
which is with respect to the additive cosets and we call it an additive cyclotomic mapping.
In this paper, we investigate the bentness of functions $f(x)={\rm Tr}_1^n(F(x))$.
For any $b\in\mathbb{F}_{p^n}$, a calculation gives
$$\begin{aligned} \widehat{f}(b)
&=\sum\nolimits_{x \in \mathbb{F}_{p^n}}\xi_p^{{\rm Tr}_1^n(F(x))-{\rm Tr}_1^n(bx)}
\\ &=\sum_{i\in\mathbb{Z}_{d}}\sum_{x\in C_i}\xi_{p}^{{\rm Tr}_1^n(a_i x^{r_i})-{\rm Tr}_1^n(bx)}
\\ &=\sum\nolimits_{i\in\mathbb{Z}_{d}}S_{i}(b),
\end{aligned}$$
where $S_{i}(b)$ is defined as in \eqref{Sij}.
Therefore, to calculate $\widehat{f}(b)$, it suffices to determine $S_{i}(b)$ for all $i\in\mathbb{Z}_{d}$.
In the case of $C=\mathbb{F}_{2^k}$ and $r_i=2^{t_i}+1$ with $k|t_i$ and $k|n$, one knows $d=2^{n-k}$ and $S_{i}(b)$ defined by \eqref{Sij} becomes
$$\begin{aligned}S_{i}(b)
&=\sum\nolimits_{x\in v_i+\mathbb{F}_{2^k}}(-1)^{{\rm Tr}_1^n(a_i x^{2^{t_i}+1})+{\rm Tr}_1^n(bx)}
\\ &=\sum\nolimits_{y \in \mathbb{F}_{2^k}}(-1)^{{\rm Tr}_1^n\big(a_i (v_i+y)^{2^{t_i}+1}\big)+{\rm Tr}_1^n(b(v_i+y))}
\\ &=(-1)^{{\rm Tr}_1^n(a_i v_i^{2^{t_i}+1}+bv_i)}\sum\nolimits_{y \in \mathbb{F}_{2^k}}(-1)^{{\rm Tr}_1^k({\rm Tr}_k^n(a_iv_i^{2^{t_i}}+a_iv_i+a_i^{2^{-1}}+b)y)},
\end{aligned}$$
which gives
$$S_{i}(b)=\left \{\begin{array}{ll}
2^k(-1)^{{\rm Tr}_1^n(a_i v_i^{2^{t_i}+1}+bv_i)}, & {\rm if} \,\, {\rm Tr}_k^n(a_i(v_i^{2^{t_i}}+v_i)+a_i^{2^{-1}}+b)=0,\\
0, &{\rm otherwise}.
\end{array} \right.$$
Then we give the Walsh transform of $f(x)$ by the following theorem.

\begin{thm}\label{thm.bent-add}
Let $n$, $k$ and $t_i$ be positive integers satisfying $k|n$ and $k| t_i$ for $i\in\mathbb{Z}_{2^{n-k}}$. Let
\[f(x)={\rm Tr}_1^n(a_ix^{2^{t_i}+1}),\,\, {\rm if} \,\, x\in v_i+\mathbb{F}_{2^k},\,i\in\mathbb{Z}_{2^{n-k}},\]
where $a_i\in\mathbb{F}_{2^n}$ and $\upsilon_i$ are the representatives of cosets of $\mathbb{F}_{2^n}$ modulo $\mathbb{F}_{2^k}$. Then for any $b\in\mathbb{F}_{2^n}$,
\[\widehat{f}(b)=2^k\sum_{i\in E(b)}(-1)^{{\rm Tr}_1^n(a_i v_i^{2^{t_i}+1}+bv_i)}\]
and $E(b)$ is defined by
\begin{equation}\label{Eb}
  E(b):=\{i\in\mathbb{Z}_{2^{n-k}}: {\rm Tr}_k^n(b)={\rm Tr}_k^n(a_i(v_i^{2^{t_i}}+v_i)+a_i^{2^{-1}})\}.
\end{equation}
\end{thm}

By using Kasami functions, we generate bent functions from a class of additive cyclotomic mappings of index $q$ over $\mathbb{F}_{q^2}$.
\begin{thm}\label{thm.bent-kasami}
Let $\xi$ be a primitive element of $\mathbb{F}_{q}$ and define $\xi^{\infty}=0$, where $q=2^m$ and $m$ is a positive integer. Define
\[f(x)={\rm Tr}_1^m(\alpha_ix^{q+1}),\,\,{\rm if} \,\, x\in N_i,\,i\in\{\infty\}\cup\mathbb{Z}_{q-1},\]
where $\alpha_{i}\in\fq$ and $N_i=\{x\in \mathbb{F}_{q^2}: x^q+x=\xi^i\}$ for $i\in\{\infty\}\cup\mathbb{Z}_{q-1}$.
If
$$\{\alpha_{i}^2\xi^{2i}+\alpha_{i}: i\in\{\infty,0,\cdots,q-2\}\}=\fq,$$
then
$f(x)$ is a bent function over $\fqt$. Moreover, for any $b\in\fqt$, the Walsh transform of $f(x)$ is
\[\widehat{f}(b)=q(-1)^{\varphi_{t}(b)}, \,\,{\rm if} \,\, b^q+b=\alpha_t\xi^t+\alpha_t^{2^{-1}},\,\,t\in\{\infty\}\cup\mathbb{Z}_{q-1}\]
with
\begin{equation}\label{phi-b}
\varphi_{t}(b)=\left \{\begin{array}{ll}
{\rm Tr}_1^m(\xi^tb), & {\rm if} \,\, \alpha_t=0,\\
{\rm Tr}_1^m(\alpha_t^{-1}b^{q+1})+1, &{\rm if} \,\, \alpha_t\ne 0.
\end{array} \right.
\end{equation}
\end{thm}

\begin{proof}
In this case, $k=t_i=m$, $\alpha_i=a_i+a_i^q$ for $i\in\mathbb{Z}_{q-1}$ and $\alpha_{\infty}=a_{q-1}+a_{q-1}^q$. Then we calculate the values of $v_i$ for $i\in\mathbb{Z}_{q}$ in Theorem~\ref{thm.bent-kasami}. Suppose that $x_0$ is a solution of the equation $x^q+x=1$. It can be verified that $\xi^ix_0$ is a solution of the equation $x^q+x=\xi^i$. This allows us to write $N_i=\{x\in \mathbb{F}_{q^2}: x^q+x=\xi^i\}=\xi^ix_0+\mathbb{F}_{q}$ for each $i\in\{\infty\}\cup\mathbb{Z}_{q-1}$. That means $v_i=\xi^ix_0$ for $i\in\mathbb{Z}_{q-1}$ and $v_{q-1}=\xi^{\infty}x_0=0$. Combining with the values of $v_i$, for any $b\in\fqt$,
$E(b)$ given by \eqref{Eb} turns into
\begin{equation}\label{Eb-kasami}
  E(b)=\{i\in\mathbb{Z}_{q}: b^q+b=\beta_i\},
\end{equation}
with
$$\begin{aligned} \beta_{i}
&={\rm Tr}_m^n(a_i(v_i^q+v_i)+a_i^{2^{-1}})
\\ &={\rm Tr}_m^n(a_i((\xi^ix_0)^q+\xi^ix_0))+(a_i+a_i^{q})^{2^{-1}}
\\ &={\rm Tr}_m^n(a_i(x_0^q+x_0))\xi^i+\alpha_i^{2^{-1}}=\alpha_i\xi^i+\alpha_i^{2^{-1}}
\end{aligned}$$
for $i\in\mathbb{Z}_{q-1}$ and
\[\beta_{q-1}={\rm Tr}_m^n(a_{q-1}(v_{q-1}^q+v_{q-1})+a_{q-1}^{2^{-1}})=(a_{q-1}+a_{q-1}^q)^{2^{-1}}
=\alpha_{\infty}\xi^{\infty}+\alpha_{\infty}^{2^{-1}}\]
due to $\alpha_i=a_i+a_i^q$ for $i\in\mathbb{Z}_{q-1}$ and $\alpha_{\infty}=a_{q-1}+a_{q-1}^q$.
Therefore $\{\beta_{i}: i\in\mathbb{Z}_{q}\}=\{\alpha_{i}\xi^{i}+\alpha_{i}^{2^{-1}}: i\in\{\infty,0,\cdots,q-2\}\}=\fq$, which implies there exist a unique $t\in\{\infty,0,\cdots,q-2\}$ such that $b^q+b=\alpha_t\xi^t+\alpha_t^{2^{-1}}$ for any $b\in\fqt$.
Moreover, in the case of $b^q+b=\alpha_t\xi^t+\alpha_t^{2^{-1}}$, $E(b)=\{t\}$ and from Theorem \ref{thm.bent-add}, one derives
$$\begin{aligned}\widehat{f}(b)=&q(-1)^{{\rm Tr}_1^n(a_t (\xi^tx_0)^{q+1}+b\xi^tx_0)}
\\ =&q(-1)^{{\rm Tr}_1^m(\alpha_t (\xi^tx_0)^{q+1}+(bx_0+(bx_0)^q)\xi^t)}
\\ =&q(-1)^{{\rm Tr}_1^m(\alpha_t \xi^{2t}x_0^{q+1}+((b+b^q)x_0+b^q)\xi^t)}
\end{aligned}$$
due to $x_0^q+x_0=1$ and $\xi\in \mathbb{F}_{q}$. Denote $\varphi_{t}(b):={\rm Tr}_1^m(\alpha_t \xi^{2t}x_0^{q+1}+((b+b^q)x_0+b^q)\xi^t)$. Then $\widehat{f}(b)=q(-1)^{\varphi_{t}(b)}$, which implies $f(x)$ is bent. Furthermore, we claim that $\varphi_{t}(b)$ is given by \eqref{phi-b}. If $\alpha_t=0$, then $b^q+b=0$ and one readily gets $\varphi_{t}(b)={\rm Tr}_1^m(\xi^tb)$.
For $\alpha_t\ne 0$, combining with the facts $b^q+b=\alpha_t\xi^t+\alpha_t^{2^{-1}}$ and $x_0^q+x_0=1$, a calculation gives
$$\begin{aligned} \varphi_{t}(b)=&{\rm Tr}_1^m\big(\alpha_t\xi^{2t}x_0(x_0+1)+\xi^t\big((\alpha_t\xi^t+\alpha_t^{2^{-1}})x_0+(\alpha_t\xi^t
+\alpha_t^{2^{-1}}+b)\big)\big)
\\ =&{\rm Tr}_1^m(\alpha_t\xi^{2t}x_0^2+\alpha_t^{2^{-1}}\xi^{t}x_0+\alpha_t\xi^{2t}+\alpha_t^{2^{-1}}\xi^{t}
+b\xi^t)
\\ =&\alpha_t^{2^{-1}}\xi^{t}(x_0^q+x_0)+\alpha_t^{2^{-1}}(\xi^{t}+\xi^{t})+
\sum\nolimits_{i=0}^{m-1}(b\xi^t)^{2^i}
\\=&\alpha_t^{2^{-1}}\xi^{t}+\sum\nolimits_{i=0}^{m-1}(b\xi^t)^{2^i}
\end{aligned}$$
On the other hand, we have
$$\begin{aligned} {\rm Tr}_1^m(\alpha_t^{-1}b^{q+1})
=&{\rm Tr}_1^m(\alpha_t^{-1}b(b+\alpha_t\xi^t+\alpha_t^{2^{-1}}))
\\ =&{\rm Tr}_1^m(\alpha_t^{-1}b^2+\alpha_t^{-2^{-1}}b+b\xi^t)
\\ =&\alpha_t^{-2^{-1}}(b^q+b)+\sum\nolimits_{i=0}^{m-1}(b\xi^t)^{2^i}
\\=&\alpha_t^{-2^{-1}}(\alpha_t\xi^t+\alpha_t^{2^{-1}})+\sum\nolimits_{i=0}^{m-1}(b\xi^t)^{2^i}
\\=&\alpha_t^{2^{-1}}\xi^{t}+\sum\nolimits_{i=0}^{m-1}(b\xi^t)^{2^i}+1.
\end{aligned}$$
Therefore, $\varphi_{t}(b)={\rm Tr}_1^m(\alpha_t^{-1}b^{q+1})+1$.
This completes the proof.
\end{proof}

\begin{remark}
If one takes $\alpha_{\infty}=\alpha_0=\cdots=\alpha_{q-2}=a\in\fq^*$, then $f(x)$ in Theorem \ref{thm.bent-kasami} is reduced to the monomial case, that is $f(x)={\rm Tr}_1^m(a x^{q+1})$. It can be verified that $a^2\xi^{2i}+a$ are $q$ distinct elements in $\fq$ when $i$ runs over $\{\infty,0,\cdots,q-2\}$. Otherwise, assume that there are distinct $i,j\in\{\infty,0,\cdots,q-2\}$ such that $a^2\xi^{2i}+a=a^2\xi^{2j}+a$, then $\xi^{2(i-j)}=1$, which is impossible since $\gcd(2,\,q-1)=1$ and $i-j\ne 0$. Thus $\{a^2\xi^{2i}+a: i\in\{\infty,0,\cdots,q-2\}\}=\fq.$ Theorem \ref{thm.bent-kasami} gives that $f(x)$ is bent with the dual function ${\rm Tr}_1^m(a^{-1}x^{q+1})+1$,
which is consistent with the result obtained by Mesnager in \cite{M}.
\end{remark}

Note that it is easy to find parameters satisfying the condition given in Theorem \ref{thm.bent-kasami} such that $f(x)$ is bent.  To show that, we provide an equivalent  condition which help us to explicitly construct bent functions. Without of loss of generality, assume that $\alpha_{\infty}\ne0$. It can be verified that the condition $\{\alpha_{i}^2\xi^{2i}+\alpha_{i}: i\in\mathbb{Z}_{q-1}\}\bigcup\{\alpha_{\infty}\}=\fq$
is equivalent to
\begin{equation}\label{kasami.con2}
\{\alpha_{\infty}^2\xi^{2i}+\alpha_{\infty}: i\in\mathbb{Z}_{q-1}\}=\{\alpha_{i}^2\xi^{2i}+\alpha_{i}: i\in\mathbb{Z}_{q-1}\}.
\end{equation}
By using it, we give the following construction of bent functions for the case $\#\{\alpha_i: i\in\{\infty\}\cup\mathbb{Z}_{q-1}\}\leq 2$. Precisely, let $\alpha_i=a$ for $i\in\mathbb{Z}\subseteq \mathbb{Z}_{q-1}$
and $\alpha_i=c\ne0$ for others $i$. Then \eqref{kasami.con2} becomes
\begin{equation}\label{kasami.con3}
\{c^2\xi^{2i}+c: i\in\mathbb{Z}\}=\{a^2\xi^{2i}+a: i\in\mathbb{Z}\}.
\end{equation}
We shall present the bent functions for two cases: $a=0$ and $a\ne0$.

In the case of $a=0$, \eqref{kasami.con3} is equivalent to $c^2\xi^{2i}+c=0$, that is, $\xi^i=c^{2^{m-1}-1}$ for a unique $i\in\mathbb{Z}_{q-1}$. This means $\alpha_i=0$ for such $i$ and $\alpha_i=c\ne0$ for others $i$. Then the following corollary can be obtained from Theorem \ref{thm.bent-kasami}.

\begin{cor}\label{cor.kasami0}
Let $q=2^m$ and $c\in\mathbb{F}_{q}^*$, where $m$ is a positive integer. Then
$$f(x)=\left \{\begin{array}{ll}
0, & {\rm if} \,\, x^q+x=c^{2^{m-1}-1},\\
{\rm Tr}_1^m(c x^{q+1}), &{\rm otherwise}
\end{array} \right.$$
is a bent function over $\mathbb{F}_{q^2}$, and its dual is
$$\widetilde{f}(x)=\left \{\begin{array}{ll}
{\rm Tr}_1^m(c^{2^{m-1}-1}x), & {\rm if} \,\, x\in\fq,\\
{\rm Tr}_1^m(c^{-1}x^{q+1})+1, &{\rm otherwise} .
\end{array} \right.$$
\end{cor}

If $a\ne 0$, then Theorem \ref{thm.bent-kasami} and \eqref{kasami.con3} give the following corollary directly.

\begin{cor}\label{cor.kasami1}
Let $q=2^m$ and $a,\,c\in\mathbb{F}_{q}^*$ satisfying \eqref{kasami.con3}. Denote $N:=\{\xi^{i}: i \in\mathbb{Z} \}$, where $\xi$ is a primitive element of $\mathbb{F}_{q}$ and $\mathbb{Z}$ is a subset of $\mathbb{Z}_{q-1}$. Then
$$f(x)=\left \{\begin{array}{ll}
{\rm Tr}_1^m(a x^{q+1}), & {\rm if} \,\,x^q+x\in N,\\
{\rm Tr}_1^m(c x^{q+1}), &{\rm otherwise},
\end{array} \right.$$
is a bent function over $\mathbb{F}_{q^2}$ and its dual is
$$\widetilde{f}(x)=\left \{\begin{array}{ll}
{\rm Tr}_1^m(a^{-1}x^{q+1})+1, & {\rm if} \,\, x^q+x=a\xi^i+a^{2^{-1}},\,i\in\mathbb{Z}\\
{\rm Tr}_1^m(c^{-1}x^{q+1})+1, &{\rm otherwise} .
\end{array} \right.$$
\end{cor}

It is clear that any $c\in\fq^*$ gives a bent function as in Corollary \ref{cor.kasami0}. Then we only give an example for the construction in Corollary \ref{cor.kasami1}.

%\begin{cor}
%Let $q=2^m$, $a,\,c\in\mathbb{F}_{q}^*$ and $\epsilon$ be a primitive element of $\mathbb{F}_{2^e}$, where $e$ and $m$ are positive integers satisfying $e|m$. Then
%$$f(x)=\left \{\begin{array}{ll}
%{\rm Tr}_1^m(a x^{q+1}), & {\rm if} \,\, x^q+x\in \mathbb{F}_{2^e}^*,\\
%{\rm Tr}_1^m(c x^{q+1}), &{\rm otherwise},
%\end{array} \right.$$
%is a bent function over $\mathbb{F}_{q^2}$ if
%\[\{a^2\epsilon^{2i}+a: i\in\mathbb{Z}_{2^e-1}\}=\{c^2\epsilon^{2i}+c: i\in\mathbb{Z}_{2^e-1}\}.\]
%Moreover, for any $b\in\fqt$, the Walsh transform of $f(x)$ is
%$$\widehat{f}(b)=\left \{\begin{array}{ll}
%q(-1)^{{\rm Tr}_1^m(a^{-1}b^{q+1})+1}, & {\rm if} \,\, b^q+b=a\epsilon^t+a^{2^{-1}},\\
%q(-1)^{{\rm Tr}_1^m(c^{-1}b^{q+1})+1}, &{\rm otherwise} .
%\end{array} \right.$$
%\end{cor}

\begin{example}
{\rm Let $q=2^4$ and $N=\{\xi^{2}\}$, where $\xi$ is a primitive element of $\mathbb{F}_{2^4}$. According to Magma, there are 14 pairs $(a_1,\,a_2)$ with $a_1\ne a_2$ such that $f(x)$ in Corollary \ref{cor.kasami1} is bent over $\mathbb{F}_{2^{8}}$. For example, take $a=\xi^{9}$ and $c=\xi^{2}$.
It can be verified that $a^2\xi^{4}+a=c^2\xi^{4}+c=1$. Corollary \ref{cor.kasami1} now establishes that
$$f(x)=\left \{\begin{array}{ll}
{\rm Tr}_1^4(\xi^{9} x^{17}), & {\rm if} \,\, x^q+x=\xi^{2},\\
{\rm Tr}_1^4(\xi^{2}x^{17}), &{\rm otherwise}
\end{array} \right.$$
is a bent function over $\mathbb{F}_{2^8}$, and its dual is
$$\widetilde{f}(b)=\left \{\begin{array}{ll}
{\rm Tr}_1^4(\xi^{6}x^{17})+1, & {\rm if} \,\, x^q+x=1,\\
{\rm Tr}_1^4(\xi^{13}x^{17})+1, &{\rm otherwise} .
\end{array} \right.$$}
\end{example}

\section{Switching between cyclotomic form and polynomial form of bent functions}\label{sec.poly}
Boolean functions $f,\,f':\mathbb{F}_{2^n}\rightarrow \mathbb{F}_2$ are extended-affine equivalent (EA-equivalent) if there exist an affine permutation $L$ of $\mathbb{F}_{2^n}$ and an affine function $l:\mathbb{F}_{2^n}\rightarrow\mathbb{F}_2$ such that $f'(x)=(f\circ L)(x)+l(x)$. A class of bent functions
is called complete if it is globally invariant under EA-equivalence and the completed version of a class is the set of all functions EA-equivalent to the functions in the class.
In this section, we first investigate the polynomial form of those bent functions proposed in Section \ref{cons1} and Section \ref{cons2}, and then study the EA-equivalence between proposed bent functions and known ones.

The switching between multiplicative cyclotomic form and polynomial form is characterized as below.

\begin{lem}{\rm(\cite{AW,W})}\label{lem.poly}
Let $F(x)$ be an index $d$ generalized cyclotomic mapping defined as in \eqref{F}.
Then the polynomial form of $F(x)$ is
\[F(x)=\frac{1}{d}\sum_{i,j=0}^{d-1}\omega^{-ij\cdot\frac{p^n-1}{d}}a_i x^{j\cdot\frac{p^n-1}{d}+r_i}.\]
 \end{lem}

Recall that $u=\omega^{(q-1)(2^{n-1}-1)}$, i.e., $\omega^{q-1}=u^{-2}$.
From Lemma \ref{lem.poly}, we give the polynomial forms of bent functions proposed in Section \ref{cons1}.

(1) Dillon case: First of all, the polynomial form of $f(x)$ proposed in Theorem \ref{thm.bent-dillon} is
\[f(x)=\sum_{i,j=0}^{q}{\rm Tr}_1^n(u^{2ij}a_i x^{(q-1)(j+l_i)}),\]
which is a Dillon type polynomial. Note that such $f(x)$ restricted to the cosets $u\fq^*$ are constant where $u$ ranges over $\mu_{q+1}$, and thus belong to $\mathcal{PS}_{ap}$ class \cite{D}.
In fact, Dillon in his thesis \cite{D} shows more precisely that, a Boolean function over $\fqt$ with the form $g(x^{q-1})$ and $g(0)=0$, is bent if and only if $g(h)=1$ for exactly $2^{m-1}$ elements in $\mu_{q+1}$.
Obviously, in Theorem \ref{thm.bent-dillon}, $f(x)=g(x^{q-1})$ with $g(x)=\sum_{i,j=0}^{q}{\rm Tr}_1^n(u^{2ij}a_i x^{j+l_i})$ and the condition $\sum\nolimits_{i\in\mathbb{Z}_{q+1}}(-1)^{{\rm Tr}_1^n(a_iu^{-2il_i})}=1$ implies that there are exactly $2^{m-1}$ $t$'s in $\mathbb{Z}_{q+1}$ such that $g(u^{t})=1$ if $u^{tl_t}$ runs over $\mu_{q+1}$ when $t$ runs over $\mathbb{Z}_{q+1}$. Thus our construction gives a subclass of bent functions of Dillon's construction.  We remark that  some generalizations of Dillon's construction were given later on, for example, see \cite{H,LL,N}. In fact, our construction can also explain some previous infinite classes of Dillon type bent polynomials. For example, from Algorithm 1 given in \cite{AW}, the cyclotomic form of a general Dillon type polynomial $f(x)={\rm Tr}_1^n(\sum_{t=0}^q \gamma_t x^{t(q-1)})$ is
\[f(x)={\rm Tr}_1^n((\sum\nolimits_{t=0}^q \gamma_t u^{-2it})),\,\,{\rm if} \,\, x\in u^i\mathbb{F}_{q}^*,\,i\in\{\infty\}\cup\mathbb{Z}_{q+1}.\]
Then Theorem \ref{thm.bent-dillon} yields $f(x)$ is bent if and only if
\[\sum\nolimits_{i\in\mathbb{Z}_{q+1}}(-1)^{{\rm Tr}_1^n(\sum\nolimits_{t=0}^q \gamma_t u^{-2it} )}=\sum\nolimits_{z\in\mu_{q+1}}(-1)^{{\rm Tr}_1^n(\sum_{t=0}^q \gamma_t z^{t})}=1,\]
which is consistent with the result presented by Li et al. in \cite{LHTK}.
Although our first construction coincides with some known ones in some sense,  the choice of parameters are more flexible in our construction. This can help to
construct new explicit infinite classes of Dillon type polynomials.
For instance, set $r=3$, then for odd $m$ and $K_m(c^{q+1})=0$, Corollary \ref{cor.bent-dillon3.1} generates the bent function with the polynomial form
\[f(x)={\rm Tr}_1^n\big(\sum\nolimits_{i=0}^{(q+1)/3-1}\big(c\epsilon x^{(3i+l_1)(q-1)}+cx^{(3i+l_2)(q-1)}\big)+c x^{l_2(q-1)}\big).\]
For $n=4$ and $n=6,\,8$, the $2$-th power coset representatives of $k(q-1)$, $k=1,\cdots,q$ modulo $q+1$ are $1$ and $1,3$ respectively, which means all Dillon type bent ponomials are of the forms ${\rm Tr}_1^n(a x^{q-1})$ over $\mathbb{F}_{2^4}$ and ${\rm Tr}_1^n(a x^{q-1}+b x^{3(q-1)})$ over $\mathbb{F}_{2^6}$ or $\mathbb{F}_{2^8}$. Li et al. \cite{LHTK} has characterized the bentness for such two type functions. For $n\geq 10$, due to the shortage of computer memory, we cannot verify the EA-equivalence between two bent functions. Thus the equivalence between the newly proposed Dillon type bent functions and the previously known ones is not clear, and we leave it to interested readers.

(2) Niho case: Secondly, the polynomial form of $f(x)$ proposed in Theorem \ref{thm.bent-Niho} is
\[f(x)=\sum_{i,j=0}^{q}{\rm Tr}_1^n(u^{2ij}a_i x^{(q-1)(j+s_i)+1}),\]
which is a Niho type polynomial. Such $f(x)$ is linear over each element of the Desarguesian spread and thus  belongs to the class $\mathcal{H}$ \cite{D}. From the polynomial perspective, our construction concludes some previous infinite classes of Niho type bent polynomials. For example, from Algorithm 1 given in \cite{AW}, the cyclotomic form of a general Niho type polynomial $f(x)={\rm Tr}_1^n(\sum_{t=1}^k \gamma_t x^{s_t(q-1)+1})$ is
\[f(x)={\rm Tr}_1^n((\sum\nolimits_{t=1}^k \gamma_r u^{-2s_ti})x),\,\,{\rm if} \,\, x\in u^i\mathbb{F}_{q}^*,\,i\in\{\infty\}\cup\mathbb{Z}_{q+1}.\]
Then Theorem \ref{thm.bent-Niho} yields $f(x)$ is bent if and only if $bz+b^qz^{-1}+\sum_{t=1}^k (\gamma_t z^{1-2s_t}+\gamma_t^q z^{2s_t-1})=0$ has $0$ or $2$ solutions in $\mu_{q+1}$ for any $b\in\fqt$, which coincides with the result presented by Leander and Kholosha in \cite{LK}.
In terms of the equivalence, Abdukhalikov \cite{A} has determined all the equivalence classes of Niho type bent functions for $m\leq 6$. For instance, Example \ref{eq.Niho} is equivalent to ${\rm Tr}_1^3(x^{36})+{\rm Tr}_1^6(x^{22})$ over $\mathbb{F}_{2^6}$. However, for $m\geq 7$, we have not yet found an efficient way to select appropriate parameters to search for new bent functions. We encourage interested readers to construct specific infinite classes of Niho type bent functions from cyclotomic mappings.

As for the additive case, the polynomial form of bent functions from additive cyclotomic mappings can be obtained using Lagrange interpolation.

(3) Kasami case: The polynomial form of $f(x)$ proposed in Theorem \ref{thm.bent-kasami} is
\begin{equation}\label{eq.poly.kasami}
f(x)=\sum_{i=0}^{q-2}{\rm Tr}_1^m(a_i x^{q+1})((x^q+x+\xi^i)^{q-1}+1)+{\rm Tr}_1^m(a_{\infty}x^{q+1})((x^q+x)^{q-1}+1).
\end{equation}
According to the equivalent condition for a bent function belonging to the Maiorana-McFarland class $\mathcal{MM}$ \cite{D,Mc} given by \cite{D}, it can be checked that $f(x)$ in \eqref{eq.poly.kasami} belongs to the $\mathcal{MM}$ class. Note that Fernando and Hou \cite{FH} also characterised the polynomial form of such type functions.

\begin{lem}{\rm(\cite{FH})}\label{lem.poly.add}
Let $\xi$ be an element of order $k$ with $k|(p^n-1)$, and define $\xi^{\infty}=0$. Let $f_{\infty},\,f_0,\cdots,f_{k-1}\in\mathbb{F}_{p^n}[x]$ and $\phi: \mathbb{F}_{p^n}\rightarrow\{\xi^i: i=\infty,0,\cdots,k-1\}$.
Then the polynomial form of
$$H(x)=f_i(x) \,\,{\rm if}\,\, \phi(x)=\xi^i,\,i\in\{\infty,0,\cdots,k-1\}$$
is
\[H(x)=f_{\infty}(x)(1-\phi(x)^{p^n-1})+\frac{1}{k}\sum_{i,j=0}^{k-1}\xi^{-ij}f_j(x) \phi(x)^i.\]
 \end{lem}

Taking $k=q-1$, $\phi(x)=x^q+x$ and $f_i(x)={\rm Tr}_1^m(a_i x^{q+1})$ for $i=\infty,0,\cdots,k-1$, then
$f(x)$ in Theorem \ref{thm.bent-kasami} coincides with $H(x)$,  which can be rewritten as
\[f(x)={\rm Tr}_1^m(a_{\infty} x^{q+1})(1+(x^q+x)^{p^n-1})+\sum_{i,j=0}^{q-2}\xi^{-ij}{\rm Tr}_1^m(a_j x^{q+1})(x^q+x)^i.\]
In fact, this is consistent with the expansion of \eqref{eq.poly.kasami}.
From a polynomial point of view, our construction produce infinite classes of bent functions with algebraic degrees higher than 2. For instance, $f(x)$ in Corollary \ref{cor.kasami0} is of the polynomial form
\begin{equation}\label{eq.Kasami.poly}
f(x)={\rm Tr}_1^m(c x^{q+1})(x^q+x+c^{2^{m-1}-1})^{q-1}.
\end{equation}
It can be verified that the algebraic degree of $f(x)$ is $m$ for $2\leq m\leq 10$ when $c=1$, which achieves the optimal algebraic degree. Next, we study the EA-equivalence between $f(x)$ in \eqref{eq.Kasami.poly} and known bent polynomials. Firstly, it can be verified by Magma that $f(x)$ is EA-inequivalent to the five classes of bent monomials. Recall from (1) and (2), when $m=3$, all known Dillon type bent polynomials are of the form ${\rm Tr}_1^6(a x^{q-1}+b x^{3(q-1)})$ and there are only two Niho type bent polynomials, ${\rm Tr}_1^3(x^{36})$ and ${\rm Tr}_1^3(x^{36})+{\rm Tr}_1^6(x^{22})$, up to equivalence. Magma shows $f(x)$ are EA-inequivalent to these three classes of bent functions. We also note that there are so many bent functions based on the second construction and thus we decide not to consider the equivalence with all known ones. Therefore we leave this problem for future study.

%Thus we cordially invite readers to attack this problem.

\section{Concluding remarks}\label{conc}
In this paper, we investigated the construction of Boolean bent functions from cyclotomic mappings. Firstly, using Dillon functions and Niho functions as the branch functions over index $q+1$ multiplicative cyclotomic cosets of $\fqt$ respectively, we obtained two generic constructions of bent functions and then derived several new explicit infinite families of bent functions. Secondly, a generic construction was presented by using Kasami functions as branch functions over index $q$ additive cyclotomic cosets of $\fqt$, from which we got some explicit constructions of bent functions. Finally, switching between the cyclotomic form and polynomial form, we showed these three classes of bent functions belong to the $\mathcal{PS}$ class, the class $\mathcal{H}$ and the $\mathcal{MM}$ class respectively.  EA-equivalence of these bent functions has been briefly discussed and it is worth pursuing a further study.

\section*{Acknowledgements}
This work was supported by the National Key Research and Development Program of China (No. 2021YFA1000600), the National Natural Science Foundation of China (No. 62072162), the Natural Science Foundation of Hubei Province of China (No. 2021CFA079), the Knowledge Innovation Program of Wuhan-Basic Research (No. 2022010801010319),  the Innovation Group Project of the Natural Science Foundation of Hubei Province of China (No. 2003AFA021), and Natural Sciences and Engineering Research Council of Canada (RGPIN-
2023-04673).

\end{document}